\documentclass[10pt]{article}

\usepackage{color, amsmath,amssymb, amsfonts, amstext,amsthm, latexsym}
\usepackage{epsfig}
\usepackage{longtable}
\usepackage{graphicx}
\usepackage{subfigure}

\renewcommand{\phi}{\varphi}

\newtheorem{theorem}{Theorem}
\newtheorem{lemma}{Lemma}

\newtheorem*{assumption}{Assumption}
\newtheorem{definition}{Definition}

\begin{document}
\title{Fokker-Planck equations for Marcus stochastic differential equations driven by L\'{e}vy processes}
 \author{\small{\it{Xu Sun$^{1,2,}$,   Xiaofan Li$^2$,    and Yayun Zheng$^1$}}\\
 \\\small{$^1$Huazhong University of Science and Technology,
  Wuhan 430074, Hubei, China}
   \\\small{$^2$Illinois Institute of Technology,
  Chicago, IL 60616, USA}\\
   {\small{E-mail:  xsun15@iit.edu; lix@iit.edu; yayzheng@gmail.com}}\\
  }

\date{April 30, 2016}
\maketitle

\pagestyle{plain}

\begin{abstract}
Marcus stochastic differential equations (SDEs) often are appropriate models for stochastic dynamical systems  driven by non-Gaussian L\'{e}vy processes  and have wide applications in engineering and physical sciences. The probability density of the solution to an SDE offers complete statistical information on the underlying stochastic process. Explicit formula for the Fokker-Planck equation, the governing equation for the probability density, is well-known when the SDE is driven by a Brownian motion. In this paper, we address the open question of finding the Fokker-Plank equations for Marcus SDEs in arbitrary dimensions   driven by non-Gaussian L\'{e}vy processes. The equations are given in a simple form that facilitates theoretical analysis and numerical computation. Several examples are presented to illustrate how the theoretical results can be applied to obtain Fokker-Planck equations for Marcus SDEs driven by L\'{e}vy processes.

\bigskip



\end{abstract}
\newpage

\section{Introduction}

Stochastic differential equations (SDEs) have been widely used as mathematical models for stochastic dynamical systems \cite{Klebaner2005,Oksendal2003,Applebaum2009}. Soutions of the SDEs, which are usually nowhere differentiable, are often interpreted   in terms of some stochastic integral. The definition of the solutions may not be unique, because it depends on the specific  stochastic integral involved.

Stochastic dynamical systems under excitation of Gaussian white noise are often modeled by SDEs driven by Brownian motions \cite{Klebaner2005,Oksendal2003}, for which there are two popular definitions, namely, It\^{o} SDEs and Stratonovich SDEs. While the former finds wide application  in finance, enconogy and biology, the latter turns out to be more appropriate models in engineering and  physical  sciences like mechanics and physics \cite{Oksendal2003,LinCai2004}. 

Stochastic dynamical systems under excitation of non-Gaussian white noise are often modeled by SDEs driven by L\'{e}vy processes. The SDEs driven by L\'{e}vy processes can be defined in sense of It\^{o} or Marcus.  The Marcus SDEs \cite{Marcus1978,Marcus1981,KurtzPardouxProtter1995,Applebaum2009} can be regarded as the generalization of Stratonovich SDEs. In addition to the Stratonovich correction term, Marcus SDEs have an extra correction term due to the jumps.  It is recently shown in \cite{SunDuanLi2013} that the DiPaola-Falsone SDEs  \cite{DiPaolaFalsone1993,DiPaolaFalsone1993b}, which are widely used in engineering and physics \cite{KanazawaSagawaHayakawa2012,WanZhu2014}, are actually equivalent to Marcus SDEs.

One of the main  tasks in the research field of stochastic dynamical systems is to quantify how uncertainty propagates and evolves. Given the SDE model of a stochastic system, there are many methods available to achieve the goal. One of them  is to study how the moments of the solution evolve in time \cite{LinCai2004}. Although the method  of moments can provide some important information for the uncertainty,  the statistical feature of the non-Gaussian distributions can not be fully captured by a finite number of moments in general. Another popular method is to obtain statistical information about the solution of the SDE path-wisely using Monte Carlo simulations \cite{KloedenPlaten1992}. However,  the efficiency of this method is significantly limited by the accuracy and convergence rate of  Monte Carlo simulation. To better quantify the uncertainty propagation and evolution, it is highly desirable to \emph{obtain the  probability density of the solution}, which contains the complete statistical information about the uncertainty.

Fokker-Planck equations, also known as Fokker-Planck-Kolmogorov equations or forward Kolmogorov equations,   are deterministic equations describing how   probability density functions evolve. For It\^{o} or Stratonovich SDEs driven by Brownian motions,  there are well established formula to write down the corresponding Fokker-Planck equations \cite{Klebaner2005,Oksendal2003,LinCai2004}. For example, consider the following It\^{o} SDE 
\begin{equation}\label{s1_000}
  {\rm d} X(t) = f(X(t)){\rm d}t + \sigma(X(t))  {\rm d} B(t), \qquad  X(0) = x_0\in \mathbb{R}^{d},
\end{equation}
where $X(t)=(X_1(t),X_2(t),\cdots,X_d(t))^T \in \mathbb{R}^d$, $f=(f_1, f_2,\cdots,f_d)^T: \mathbb{R}^d\to \mathbb{R}^d$, $\sigma=(\sigma_{ij})_{d\times n}: \mathbb{R}^d \to \mathbb{R}^{d\times n}$. $B(t)$ is an $n$-dimensional Brownian motion,  and $f$ and $g$ satisfy certain smoothness conditions.    The probability density function $p(x,t)$ for the solution $X(t)$ in (\ref{s1_000}) can be expressed as \cite{Klebaner2005}
\begin{align}\label{s1_001}
  \frac{\partial p(x, t)}{\partial t} &= - \sum^{d}_{i=1}\frac{\partial }{\partial x_{i}} \left[f_{i}(x)  p(x, t)\right] + \sum^{d}_{i,j=1} \frac{\partial ^2}{\partial x_{i}\partial x_{j}}\left[D_{ij}(x)p(x, t) \right],
   \end{align}
where $D_{ij}(x)= \sum_{k=1}^n \sigma_{ik}(x)\sigma_{kj}(x)$.

However, up to date,  there is no formula available for the  Fokker-Plank equations corresponding to   SDEs driven by general non-Gaussian L\'{e}vy processes. The reason lies in that Fokker-Planck equations require explicit expressions for the adjoint of the infinitesimal generator associated with  the solution of the SDEs. These explicit forms, however,  are often unavailable for SDEs driven by non-Gaussian L\'{e}vy processes. For It\^{o} SDEs driven by  L\'{e}vy processes, although the general result is unknown, there are some special cases where the corresponding Fokker-Planck equations are  available, e.g.,   Schertzer {\it{et al.}}  \cite{Schertzer2001} derived Fokker-Planck equations for It\^{o} SDEs  driven by $\alpha$-stable processes. There is little existing result for  Fokker-Planck equations for Marcus SDEs. The only published result we can find so far is given by \cite{SunDuan2012}, where the authors have derived Fokker-Plank equations for one-dimensional Marcus SDEs driven by scalar L\'{e}vy processes under the very stringent condition that the noise coefficient is strictly nonzero. However, for Marcus SDEs,  driven by multi-dimensional L\'{e}vy processes and with more general noise coefficients, a form of  Fokker-Planck equations that is accessible for computation still remains  an open problem. 

Fokker-Planck equations have been a widely-used important tool to quantify the propagation of uncertainty in stochastic dynamical systems  driven by Brownian motions\cite{Risken1996,LinCai2004}. In contrast, the unavailability of the Fokker-Planck equations for Marcus SDEs driven by L\'{e}vy processes poses as a significant obstacle on quantifying the uncertainty in stochastic dynamical systems under excitation of non-Gaussian L\'{e}vy noise.  We note that since the Fokker-Planck equations are not available for Marcus SDEs with non-Gaussian L\'{e}vy noise, effort has been made to obtain approximate Fokker-Planck equations by stochastic averaging under the condition of small parameters, see  \cite{ZengZhu2010c} among others.

The \emph{main objective of this paper} is to derive the  Fokker-Planck equations for   multi-dimensional Marcus SDEs  driven by L\'{e}vy processes.

Let $L(t)$ denotes the $\mathbb{R}^n$-valued L\'{e}vy process characterized by the generating triplet ($b, A, \nu$), where $b$ is a vector in $\mathbb{R}^{n}$, $A=(A_{ij})$ is a positive definite  $n$-by-$n$ matrix, and $\nu$ is a measure defined on $\mathbb{R}^{n}\backslash\{0\}$  satisfying 
\begin{align}\label{s1_1}
  \int_{\mathbb{R}^{n}\backslash \{0\}}( |y  |^2 \wedge 1) \nu({\rm d}y) < \infty.
\end{align}
Here $y = (y_{1}, y_{2}, \cdots, y_{n})$ is a vector in $\mathbb{R}^{n}$, $ |y | = \sqrt{y_1^2+y_2^2+\cdots+y_n^2}$ is the usual Euclidean norm of $y$, and the operation '$\wedge$' is defined as $a\wedge b  = {\rm {min}}\{a,b\}$.
By L\'{e}vy-It\^{o} decomposition \cite{Applebaum2009}, the L\'{e}vy process $L(t)$ can be expressed as
\begin{equation}\label{s1_3}
  L(t) = bt + B(t) + \int_{ |y |< 1}y \tilde{N}(t, {\rm d}y) + \int_{ |y |\geq 1}y  N(t, {\rm d}y),
\end{equation}
where $B(t)$ is the Brownian motion with the covariance matrix $A$, and $N(t, {\rm d}y)$ is the Poisson random measure defined as
\begin{equation}\label{s1_4}
  N(t, Q)(\omega) =\#\{s\big| 0\leq s \leq t;\Delta L(s)(\omega)\in Q\},
\end{equation}
with $\#\{\cdot\}$ representing the number of elements in the set '$\cdot$'. $Q$ is a Borel set in $\mathcal{B}(\mathbb{R}\backslash\{0\})$,  $\Delta L(t)$  the jumps of $L(t)$ at time $t$ defined as $\Delta L(t) = L(t) - L(t-)$, and $\tilde N({\rm d}t, {\rm d}y)$ is the compensated Poisson measure defined as $\tilde N({\rm d}t, {\rm d}y) = N({\rm d}t, {\rm d}y)-{\rm d}t\, \nu({\rm d}y)$.

Each component $L_j(t)$ ($j=1,2,\cdots,n$)  of $L(t)$ can be expressed as
\begin{equation}\label{s1_5}
  L_{j}(t) = b_{j}t + \sum_{k}^{n}\tau_{jk}B_{k}(t) + \int_{ |y |< 1}y_j \tilde{N}(t, {\rm d}y) + \int_{ |y |\geq 1}y_j  N(t, {\rm d}y), \quad \quad 1 \leq j \leq n,
\end{equation}
 where $b_{j}$ is the $j$-th component of vector $b$, $B_{k}(t)$ ($k=1,2,\cdots,n$) are independent  standard scalar Brownian motions and $\tau=(\tau_{ij})$ is a $n$-by-$n$ real matrix, which is related to the the covariance matrix $A$ by $A=(A_{ij})=\left(\sum\limits_{k=1}^{n} \tau_{ik} \tau_{kj}\right)$. Note that, throughout of this paper, if $x$ represent an element in $\mathbb{R}^n$, then $x_j$ will be used to denote the $j$-th component of $x$ without further claim.

Consider the SDE  in sense of Marcus,
\begin{equation}\label{s1_6}
  {\rm d} X(t) = f(X(t)){\rm d}t + \sigma(X(t)) \diamond {\rm d} L(t), \qquad  X(0) = x_0\in \mathbb{R}^{d},
\end{equation}
where $X(t)=(X_1(t),X_2(t),\cdots,X_d(t))^T \in \mathbb{R}^d$, $f=(f_1, f_2,\cdots,f_d)^T: \mathbb{R}^d\to \mathbb{R}^d$, $\sigma=(\sigma_{ij})_{d\times n}: \mathbb{R}^d \to \mathbb{R}^{d\times n}$. The solution to the SDE (\ref{s1_6}) is interpreted as
\begin{align}\label{s1_6a}
X(t) =X(0) + \int_0^t f(X(s) ){\rm d}s+ \int_0^t \sigma(X(s-))\diamond {\rm d} L(s),
 \end{align}
where $X(s-)$ is the left limit $\lim\limits_{u<s, u\to s} X(u)$, and "$\diamond$" indicates Marcus integral \cite{Marcus1978,Marcus1981,KurtzPardouxProtter1995,Applebaum2009} defined by
\small
\begin{align}\label{s1_6b}
 \int_0^t \sigma(X_{s-})\diamond {\rm d} L(s)&=
\int_0^t \sigma(X_{s-})  {\rm d}L(s) + \frac{1}{2} \int_0^t \tilde \sigma(X(s-) )  {\rm d} s\nonumber\\
 &\quad + \sum_{0\leq s \leq t} \left[ H(\Delta L(s),  X(s-)) -X(s-) -\sigma(X(s-))\Delta L(s) \right].
 \end{align}
 \normalsize
Here $\tilde \sigma(X(s-))$ is a vector in $\mathbb{R}^d$ with the $i$-th element as
 \begin{align}\label{s1_6c}
 \tilde \sigma_i(X(s-)) = \sum\limits_{m=1}^{d} \sum\limits_{j=1}^{n}\sum\limits_{l=1}^{n} \sigma_{ml} (X(s-))\frac{\partial}{\partial x_m} \sigma_{ij} (X(s-)) A _{lj},\quad i=1,2,\cdots,d.
 \end{align}
 $H: \mathbb{R}^d\times \mathbb{R}^n\to \mathbb{R}^d$ is defined such that, for any $u\in \mathbb{R}^d$ and $v\in \mathbb{R}^n$, $H(u,v)=\Phi(1)$ with $\Phi: \mathbb{R}\to \mathbb{R}^d, r\mapsto \Phi(r)$ being the solution to the initial value problem of the system of ordinary differential equations (ODEs)
 \begin{align}\label{s1_6d}
\dfrac{{\rm d}\Phi(r)}{{\rm d}r}  = \sigma( \Phi (r)) v,\quad
\Phi(0) =u.
\end{align}
It can be written in the form of components as
  \begin{align}\label{s1_6e}
\begin{cases}
\dfrac{{\rm d}\Phi_i(r)}{{\rm d}r}  =\sum\limits_{j=1}^{n} \sigma_{ij}( \Phi_1(r),\Phi_1(r),\cdots,\Phi_d(r)) v_j, \\
\Phi(0) =u_i,\quad\quad \quad i=1,2,\cdots,d.
\end{cases}
\end{align}
The first term in the right hand side of (\ref{s1_6b}) is the stochastic integral in sense of It\^{o}. The second term is the correction term due to the continuous part of $X(t)$, which is actually the correction term due to Stratonovich integral. The last term is the correction term due to the jumps of $X(t)$. Note that Marcus integral reduces to Stratonovich integral when jumps are absent.

The goal of this paper is to derive the  Fokker-Planck equations for the multi-dimensional Marcus SDEs  (\ref{s1_6}) driven by L\'{e}vy processes. Sections of this paper are organized as follows. We introduce the main results in section 2, give the proof of the main results  in section 3, and present some examples in section 4 to illustrate how the main results is  applied to obtain Fokker-Planck equations in some specific applications.

\section{Main Result}

Let $p(x,t\big | X(0)=x_0)$ represent the probability density function for the solution $X(t)$ of the SDE (\ref{s1_6}), and for convenience, we drop  the initial condition and simply denote it by $p(x,t)$.

Throughout this work, we assume the following.
\begin{assumption}[H1]
\quad Probability density function $p(x,t)$ for the solution $X(t)$ defined in (\ref{s1_6}) exists  and  is continuously differentiable with respect to $t$ and twice continuously differentiable with respect to $x$ for $t\in \mathbb{R}$ and $x \in \mathbb{R}^d$.
\end{assumption}
\begin{assumption}[H2]
$f: \mathbb{R}^d \to \mathbb{R}^d$ is continuously differentiable (i.e., $f\in C^1(\mathbb{R}^d, \mathbb{R}^d)$) and $\sigma: \mathbb{R}^d \to \mathbb{R}^{d\times n}$ is Lipschitz and twice continuously differentiable (i.e., $\sigma\in C^2(\mathbb{R}^d, \mathbb{R}^{d\times n}))$.
\end{assumption}

The study of conditions for existence and regularity of the probability density for the solution of SDE (\ref{s1_6}) is  out of the scope of the current paper. We note that existence and regularity of probability densitys for SDEs driven by L\'{e}vy processes are currently active research area. See \cite{PriolaZabczyk2009,BodnarchukKulik2011,SongZhang2015} among others.

Our form of Fokker-Planck equation requires the mapping $\tilde H$ defined below.
\begin{definition}
For $u=(u_1, u_2, \cdots,u_d)^T \in \mathbb{R}^d$ and $v=(v_1, v_2, \cdots,v_n)^T \in \mathbb{R}^n$,  we introduce the mapping  $\tilde H$ that relies on $\sigma=(\sigma_{ij})$, the coefficient of the noise term in the SDE (\ref{s1_6}) such that
\begin{align}\label{s1_7a}
\tilde H: \mathbb{R}^d \times \mathbb{R}^n \to \mathbb{R}^d,  \quad(u,v)\mapsto \tilde H(u,v) = \Psi(1),
\end{align}
where   $\Psi:\mathbb{R}\to\mathbb{R}^d, r\mapsto\Psi(r)$  is the solution of the ODEs
  \begin{align}\label{s1_8a}
\dfrac{{\rm d}\Psi(r)}{{\rm d}r}  = -\sigma( \Psi (r)) v, \quad
\Psi(0) =u.
\end{align}
\end{definition}

\begin{theorem}[Main result]\label{theorem1}
Suppose the assumptions $H1$ and $H2$ hold, then the probability density function $p(x,t)$ for the solution $X(t)$ to the SDE (\ref{s1_6})  satisfies the following equation
\small
\begin{align}\label{s1_9}
  \frac{\partial p(x, t)}{\partial t} &= - \sum^{d}_{i=1}\frac{\partial }{\partial x_{i}}\left[\left(f_{i}(x)  + \sum^{n}_{j=1}\sigma_{ij}(x)b_{j}  + \frac{1}{2}\sum^{d}_{ m=1}\,\sum^{n}_{j,\,l=1}\frac{\partial \sigma_{ij}(x)}{\partial x_{m}}\sigma_{ml}(x)A_{lj}\right)p(x, t)\right] \nonumber\\
   & \quad + \frac{1}{2}\sum^{d}_{i,m=1}\sum^{n}_{j,\,l=1}\frac{\partial^{2}}{\partial x_{i}\partial x_{m}}[\sigma_{ij}(x)\sigma_{ml}(x)A_{jl}p(x, t)] \notag\\
   &\quad + \int_{\mathbb{R}^{n}\backslash \{0\}}\left[p(\tilde {H}(x,y), t) \left|\dfrac{\partial \tilde H(x, y)}{\partial x}\right| - p(x, t) + \sum^{d}_{i=1}\sum^{n}_{j=1}y_{j}I_{|y| < 1}(y)\frac{\partial }{\partial x_{i}}(\sigma_{ij}(x)p(x, t))\right]\nu(dy),
   \end{align}
   \normalsize
   where $(b, A, \nu)$ is the generating triplet of the L\'{e}vy process $L(t)$,  $\tilde H$ is defined in (\ref{s1_7a}) and (\ref{s1_8a}),  and $\left|\dfrac{\partial \tilde H(x, y)}{\partial x}\right|$  is the Jacobian of $\tilde H (x,y)$ with respect to $x$, defined as
\begin{align}\label{sp_9aa}
 \left|\dfrac{\partial \tilde H(x, y)}{\partial x}\right| =
\left|
  \begin{array}{ccc}
    \dfrac{\partial \tilde H_1(x, y)}{\partial x_{1}} & \cdots & \dfrac{\partial\tilde H_1(x, y)}{\partial x_{d}}\\
    \cdots & \cdots & \cdots \\
    \dfrac{\partial \tilde H_d(x, y)}{\partial x_{1}} & \cdots & \dfrac{\partial \tilde H_d(x, y)}{\partial x_{d}}\\
  \end{array}
\right|.
\end{align}
\end{theorem}

\section{Proof of the main result}

The following lemma reveals the relationship between the mapping $H$, as  defined  in (\ref{s1_6b}) and (\ref{s1_6d}), and $\tilde H$, as defined in (\ref{s1_7a}) and (\ref{s1_8a}).
\begin{lemma}\label{lemma1}
If  $\sigma: \mathbb{R}^d \to \mathbb{R}^{d\times n}$ is Lipschitz continuous and continuously differentiable (i.e., $\sigma\in C^1 (\mathbb{R}^d,\mathbb{R}^{d\times n})$), then\\
(i) for $\forall v\in \mathbb{R}^n$,  $H(\cdot,v): \mathbb{R}^d \to \mathbb{R}^d$ is invertible with the inverse mapping as $\tilde H (\cdot,v): \mathbb{R}^d \to \mathbb{R}^d$, i.e., for all $u \in \mathbb{R}^d$ and $v \in \mathbb{R}^n$,
\begin{align}\label{s1_8b}
\tilde H(H(u,v),v)=  H(\tilde H(u,v),v) =u;
\end{align}
(ii) both $H(\cdot,v)$ and $\tilde H(\cdot,v)$ are continuously differentiable.
\end{lemma}

\begin{proof}[Proof of Lemma \ref{lemma1}]
Note that $\sigma$ is Lipschitz continuous, then the solution of the ODE (\ref{s1_6d}) or (\ref{s1_8a}) exists globally and is unique.

Let $\xi: \mathbb{R}\times\mathbb{R}^d\to \mathbb{R}^d, (r,u)\mapsto \xi(r,u)=\Phi(r)$ represent the flow of ODE (\ref{s1_6d}) and $\xi_r: \mathbb{R}^d \to \mathbb{R}^d, u\mapsto \xi(r,u)$ be the time-$r$ mapping. The uniqueness of the ODE solution implies that
\begin{align}\label{s1_8c}
\xi_0=u, \quad\xi_{r_1+r_2}=\xi_{r_1}\circ\xi_{r_2}, \quad \forall r_1, r_2 \in \mathbb{R}.
\end{align}
(\ref{s1_8c}) indicates that time-$u$ mapping $\xi_r$ is invertible and $\xi_r^{-1}=\xi_{-r}$. Moreover, it is well known that the time-$r$ mapping associated with (\ref{s1_6d}) or (\ref{s1_8a}) is as smooth as $\sigma$, which follows from the dependence of the ODE solution on the initial value \cite{Chicone2006}.

Since for all $v \in \mathbb{R}^n$, $H(\cdot,v)$ and $\tilde H(\cdot,v)$ are actually time-$1$ mapping associated with (\ref{s1_6d}) and (\ref{s1_8a}), respectively, and  the flow for (\ref{s1_8a}) is exactly the time reversal version of that for (\ref{s1_6d}), we get the Lemma immediately.
\end{proof}

The following Lemma can be found in many textbooks on theory of distributions. See  \cite{Bhattacharyya2012} among others.
\begin{lemma}\label{lemma2}
Suppose $\gamma_1 \in C^0 (\mathbb{R}^n)$ and $\gamma_2 \in C^0(\mathbb{R}^n)$, if $\forall \phi\in C_0^\infty(\mathbb{R}^n$, $\int_{\mathbb{R}^n}\phi(x) \gamma_1(x)\,{\rm d}x = \int_{\mathbb{R}^n}\phi(x) \gamma_2(x)\,{\rm d}x$, then $\forall x \in \mathbb{R}^n$, $\gamma_1(x) = \gamma_2(x)$.
\end{lemma}

\begin{proof}[Proof of Theorem \ref{theorem1}]

It follows from (\ref{s1_3}) and (\ref{s1_6})  that, for $1\le i \le d$,
\begin{align}\label{sp_1}
  dX_{i}(t) &= f_{i}(X(t))\,{\rm d}t + \sum^{n}_{j=1}\sigma_{ij}(X(t-))b_{j}\,{\rm d}t +  \sum^{n}_{j,k=1}\sigma_{ij}(X(t-))\tau_{jk}\,{\rm d} B_{k}(t)\notag\\
   &\quad + \frac{1}{2}\sum^{n}_{j,l=1}\sum^{d}_{m=1}\sigma_{ml}(X(t-))A_{lj}\dfrac{\partial }{\partial x_{m}}\sigma_{ij}(X(t-))dt\notag\\
   &\quad+ \int_{|y|<1} [H_{i}(X(t-), y) - X_{i}(t-))]\tilde{N}({\rm d}t, \,{\rm d}y)\notag\\
   &\quad+ \int_{|y|\geq 1} [H_{i}(X(t-), y) - X_{i}(t-))]N(\,{\rm d}t, \,{\rm d}y)\notag\\
   &\quad+ \int_{|y|<1} [H_{i}(X(t-), y) - X_{i}(t-)) - \sum^{n}_{j=1}\sigma_{ij}(X(t-))y_{j}]\nu(\,{\rm d}y)\,{\rm d}t.
\end{align}
By using  (\ref{sp_1}) and the It\^{o} formula, $\forall \varphi\in C_0^\infty(\mathbb{R}^d)$, we can get
\begin{align}\label{sp_2}
 &  \varphi(X(t+ \Delta t)) - \varphi(X(t))\notag\\
   = & \int^{t+ \Delta t}_{t}\sum^{d}_{i=1}\left(\dfrac{\partial }{\partial x_i} \varphi(X(s-)\right)f_{i}(X(s))\,{\rm d}s +  \int^{t+ \Delta t}_{t}\sum^{d}_{i=1}\sum^{n}_{j=1}\left( \dfrac{\partial }{\partial x_i} \varphi(X(s-))\right)\left(\sigma_{ij}(X(s-))b_{j}\right)\,{\rm d}s \notag\\
   + &  \int^{t+\Delta t}_{t}\sum^{d}_{i=1}\left( \dfrac{\partial}{\partial x_i}  \varphi(X(s-))\right) \sum^{n}_{j,k=1}\sigma_{ij}(X(s-))\tau_{jk}\;{\rm d}B_{k}(s)\notag\\
   + & \dfrac{1}{2}\int^{t+\Delta t}_{t}\sum^{d}_{i,m=1}\sum^{n}_{j,l=1}\left(\dfrac{\partial }{\partial x_{i} } \varphi(X(s-))\right) \left(\dfrac{\partial}{\partial x_{m}}  \sigma_{ij}(X(s-))\right)\sigma_{ml}(X(s-)) A_{lj}\,{\rm d}s\notag \\
   + & \frac{1}{2}\int^{t+\Delta t}_{t}\sum^{d}_{i,m=1}\sum^{n}_{j,l=1}\left(\dfrac{\partial^{2}}{\partial x_{i} \partial x_{m}} \varphi(X(s-)\right) \sigma_{ij}(X(s-) \sigma_{ml}(X(s-))A_{jl}\,{\rm d}s\notag\\
   + & \int^{t+\Delta t}_{t}\int_{ |y |<1} \left[\varphi(H( X(s-),y)) - \varphi(X(s-))\right]\tilde{N}({\rm d}s, \,{\rm d}y)\notag\\
   + &  \int^{t+\Delta t}_{t}\int_{ |y |\geq 1} \left[\varphi(H( X(s-),y)) - \varphi(X(s-))\right]N({\rm d}s, {\rm d}y)\notag\\
   + & \int^{t+\Delta t}_{t}\int_{ |y |<1} \left[ \varphi(H( X(s-),y))- \varphi(X(s-)) - \sum^{d}_{i=1}\left(\dfrac{\partial }{\partial x_{i} } \varphi(X(s-))\right) \sigma_{ij}(X(s-))y_{j}\right]\nu({\rm d}y){\rm d}s.
\end{align}
Taking expectation at both sides of  (\ref{sp_2}), we get
\begin{align}\label{sp_3}
   & \int_{\mathbb{R}^{d}}\varphi(x)p(x, t+\Delta t) \,{\rm d}x - \int_{\mathbb{R}^{d}}\varphi(x)p(x, t)\,{\rm d}x \notag\\
  = & \int^{t+\Delta t}_{t}\int_{\mathbb{R}^{d}}\sum^{d}_{i=1}\dfrac{\partial \varphi(x)}{\partial x_{i}}f_{i}(x)p(x, s)\,{\rm d}x\,{\rm d}s + \int^{t+\Delta t}_{t}\int_{\mathbb{R}^{d}}\sum^{d}_{i=1}\dfrac{\partial \varphi(x)}{\partial x_{i}}\left(\sum^{n}_{j=1}\sigma_{ij}(x)b_{j}\right)p(x, s)\,{\rm d}x\,{\rm d}s \notag\\
+ & \frac{1}{2}\int^{t+\Delta t}_{t}\int_{\mathbb{R}^{d}}\sum^{d}_{i,m=1}\sum^{n}_{j,l=1}\dfrac{\partial \varphi(x)}{\partial x_{i}} \dfrac{\partial \sigma_{ij}(x)}{\partial x_{m}}\sigma_{ml}(x)A_{lj} p(x,s) \,{\rm d}x\,{\rm d}s\notag \\
+ & \frac{1}{2}\int^{t+\Delta t}_{t}\int_{\mathbb{R}^{d}}\sum^{d}_{i,m=1}\sum^{n}_{j,l=1}\dfrac{\partial^{2}\varphi(x)}{\partial x_{i}\partial x_{m}} \sigma_{ij}(x)\sigma_{ml}(x)A_{jl} p(x,s)\,{\rm d}x\,{\rm d}s\notag\\
+ & \int^{t+\Delta t}_{t} \int_{\mathbb{R}^{d}} \int_{|y|\geq 1}\left[\varphi(H(x,y)) - \varphi(x)\right]p(x,s) \,\nu({\rm d}y)\,{\rm d}x \,{\rm d}s\notag\\
+ & \int^{t+\Delta t}_{t} \int_{\mathbb{R}^{d}} \int_{|y|<1}  \left[\varphi(H(x,y)) - \varphi(x) - \sum^{d}_{i=1} \sum^{n}_{j=1}\frac{\partial \varphi(x)}{\partial x_{i}} \sigma_{ij}(x)y_{j} \right]p(x, s)\,\nu({\rm d}y)\,{\rm d}x\,{\rm d}s.
\end{align}
To obtain (\ref{sp_3}), we have changed the orders of integrals  justified by Fubini's theorem, since $\varphi\in C_0^\infty(\mathbb{R})$. Moreover, we rely on the following facts for getting (\ref{sp_3}),
\begin{equation}\label{sp_4}
  \mathbf{E}\left\{\int^{t+\Delta t}_{t}\sum^{d}_{i=1} \sum^{n}_{j,k=1}\frac{\partial}{\partial x_{i} } \left( \varphi(X(s-))\right) \sigma_{ij}(X(s-) \tau_{jk}\,{\rm d}B_{k}(s))\right\} = 0,
\end{equation}
\begin{equation}\label{sp_5}
 \mathbf{ E} \left\{\int_{|y|<1} [\varphi(H(X(s-), y)) - \varphi(X(s-))]\tilde{N}({\rm d}s, {\rm d}y)\right\} = 0,
\end{equation}
\begin{equation}\label{sp_5a}
 \mathbf{ E} \left\{\int_{|y|\ge 1} [\varphi(H(X(s-), y)) - \varphi(X(s-))]\tilde{N}({\rm d}s, {\rm d}y)\right\} = 0,
\end{equation}
and
\begin{align}\label{sp_6}
   & \mathbf{E}\left\{\int^{t+\Delta t}_{t}\int_{|y|\geq 1} \left[\varphi(H(X(s-), y)) - \varphi(X(s-))\right]\,N({\rm d}s, {\rm d}y)\right\}\notag\\
  = & \mathbf E\left\{\int^{t+\Delta t}_{t}\int_{|y|\geq 1}\left[\varphi(H(X(s-), y)) - \varphi(X(s-)) \right] \, \nu({\rm d}y)\,{\rm d}s\right\}\notag\\
= & \int^{t+\Delta t}_{t}\int_{\mathbb{R}^{d}}\int_{|y|\geq 1} [\varphi(H(X(s-), y)) - \varphi(x)]p(x, t) \,\nu({\rm d}y)\,{\rm d}x\, {\rm d}s.
\end{align}
The first identity in (\ref{sp_6}) follows from (\ref{sp_5a}).

Dividing both sides of (\ref{sp_3}) by $\Delta t$, and taking the limit of $\Delta t\to 0$, we deduce that
\begin{align}\label{sp_7}
   & \int_{\mathbb{R}^{d}}\varphi(x)\frac{\partial p(x, t)}{\partial t}\,{\rm d}x  \notag\\
  = & \int_{\mathbb{R}^{d}}\sum^{d}_{i=1}\frac{\partial \varphi(x)}{\partial x_{i}}f_{i}(x)p(x, t)dx + \int_{\mathbb{R}^{d}}\sum^{d}_{i=1} \frac{\partial \varphi(x)}{\partial x_{i}}\left(\sum^{n}_{j=1}\sigma_{ij}(x)b_{j}\right)p(x, t)\,{\rm d}x \notag\\
+ & \frac{1}{2}\int_{\mathbb{R}^{d}}\sum^{d}_{i,m=1}\frac{\partial \varphi(x)}{\partial x_{i}}\left(\sum^{n}_{j,l=1}\frac{\partial \sigma_{ij}(x)}{\partial x_{m}}\sigma_{ml}(x)A_{lj}\right)p(x, t)\,{\rm d}x\notag \\
+ & \frac{1}{2}\int_{\mathbb{R}^{d}}\sum^{d}_{i,m=1}\frac{\partial^{2}\varphi(x)}{\partial x_{i}\partial x_{m}}\left(\sum^{n}_{j,l=1}\sigma_{ij}(x)\sigma_{ml}(x)A_{jl}\right)p(x, t)\,{\rm d}x\notag\\
+ & \int_{\mathbb{R}^{d}}\int_{\mathbb{R}^{n}\backslash \{0\}}\left[\varphi(H(x,y)) - \varphi(x) - \sum^{d}_{i=1}\frac{\partial \varphi(x)}{\partial x_{i}}\left(\sum^{n}_{j=1}\sigma_{ij}(x)y_{j}\right)I_{|y| < 1}(y)\right]p(x, t)\,\nu(\,{\rm d}y)\,{\rm d}x.
\end{align}
 Using integration by parts,  we rewrite the first four terms at the right-hand side (RHS) of Eq.(\ref{sp_7})
\begin{align}\label{sp_8}
   & \int_{\mathbb{R}^{d}}\sum^{d=1}_{i}\frac{\partial \varphi(x)}{\partial x_{i}}f_{i}(x)p(x, t)\,{\rm d}x + \int_{\mathbb{R}^{d}}\sum^{d}_{i=1}\frac{\partial \varphi(x)}{\partial x_{i}}\left(\sum^{n}_{j=1}\sigma_{ij}(x)b_{j}\right)p(x, t)\,{\rm d}x \notag\\
+ & \frac{1}{2}\int_{\mathbb{R}^{d}}\sum^{d}_{i,m=1}\frac{\partial \varphi(x)}{\partial x_{i}}\left(\sum^{n}_{j,l=1}\frac{\partial \sigma_{ij}(x)}{\partial x_{m}}\sigma_{ml}(x)A_{lj}\right)p(x, t)dx\notag  \\
  = & -\int_{\mathbb{R}^{d}}\varphi(x)  p(x, t) \sum^{d}_{i=1}\frac{\partial }{\partial x_{i}}\left(f_{i}(x) + \sum^{n}_{j=1}\sigma_{ij}(x)b_{j} + \frac{1}{2}\sum^{n}_{j,l=1}\sum^{d}_{m=1}\frac{\partial \sigma_{ij}(x)}{\partial x_{m}}\sigma_{ml}(x)A_{lj}\right)\,{\rm d}x,
\end{align}
and
\begin{align}\label{sp_9}
   & \frac{1}{2}\int_{\mathbb{R}^{d}}\left(\sum^{d}_{i,m=1}\sum^{n}_{j,l=1}\frac{\partial^{2}\varphi(x)}{\partial x_{i}\partial x_{m}} \sigma_{ij}(x)\sigma_{ml}(x)A_{jl}\right)p(x, t)\,{\rm d}x \notag\\
  = & \frac{1}{2}\int_{\mathbb{R}^{d}}\varphi(x)\sum^{d}_{i,m=1}\sum^{n}_{j,l=1} \frac{\partial^{2}}{\partial x_{i}\partial x_{m}} \left( \sigma_{ij}(x)\sigma_{ml}(x)A_{jl}\,p(x, t) \right)\,{\rm d}x.
\end{align}
The last integral in the RHS of (\ref{sp_7}) becomes
\begin{align}\label{sp_10}
   & \int_{\mathbb{R}^{d}}\,{\rm d}x \int_{\mathbb{R}^{n}\backslash \{0\}}\left[\varphi(H(x,y)) - \varphi(x) - \sum^{d}_{i}\frac{\partial \varphi(x)}{\partial x_{i}}\left(\sum^{n}_{j}\sigma_{i,j}(x)y_{j}\right)I_{|y| < 1}(y)\right]p(x, t)\nu(\,{\rm d}y) \notag \\
  = & \int_{\mathbb{R}^{n}\backslash \{0\}}\,\nu(\,{\rm d}y)\int_{\mathbb{R}^{d}} \left[\varphi(H(x,y)) - \varphi(x) - \sum^{d}_{i}\frac{\partial \varphi(x)}{\partial x_{i}}\left(\sum^{n}_{j}\sigma_{i,j}(x)y_{j}\right)I_{|y| < 1}(y)\right]p(x, t)\,{\rm d}x.
\end{align}

 Let $z = H(x, y)$. By Lemma 1, we have  $x=\tilde{H}(z, y)$. Making the changing the variable $x=\tilde H(z,y)$,  we have
\begin{align}\label{sp_11}
    \int_{\mathbb{R}^{d}}\varphi(H(x, y))p(x, t)\,{\rm d}x & = \int_{\mathbb{R}^{d}}\varphi(z)p(\tilde{H}(z, y), t) \left|\dfrac{\partial \tilde H(z, y)}{\partial z}\right| \,{\rm d}z,
\end{align}
where $ \left |\dfrac{ \partial \tilde H(z,y)}{\partial z}\right |$ is the Jacobian of $\tilde H(z,y)$ with respect to $z$. Also, we have
\begin{align}\label{sp_13}
   &  \int_{\mathbb{R}^{d}} \sum^{d}_{i=1}\frac{\partial \varphi(x)}{\partial x_{i}}\left(\sum^{n}_{j=1}\sigma_{ij}(x)y_{j}I_{|y| < 1}(y)\right)p(x, t)\,{\rm d}x \notag \\
  = & -\int_{\mathbb{R}^{d}} \varphi(x)\sum^{d}_{i=1}\frac{\partial }{\partial x_{i}}\left[\sum^{n}_{j=1}\sigma_{ij}(x)y_{j}I_{|y| < 1}(y)p(x, t)\right]\,{\rm d}x.
\end{align}
Substituting  (\ref{sp_11}) and (\ref{sp_13}) into  (\ref{sp_10}), we get
\small
\begin{align}\label{sp_14}
   & \int_{\mathbb{R}^{d}}\,{\rm d}x \int_{\mathbb{R}^{n}\backslash\{0\}}\left[\varphi(H(x,y)) - \varphi(x)  - \sum^{d}_{i=1}\frac{\partial \varphi(x)}{\partial x_{i}}\left(\sum^{n}_{j=1}\sigma_{ij}(x)y_{j}\right)I_{|y| < 1}(y)\right]p(x, t)\,\nu({\rm d}y) \notag \\
  = & \int_{\mathbb{R}^{d}}\,{\rm d}x \int_{\mathbb{R}^{n}\backslash \{0\}}\varphi(x)\left[p(\tilde{H}(x, y), t) \left|\dfrac{\partial \tilde H(x, y)}{\partial x}\right| - p(x, t) + \sum^{d}_{i=1}\frac{\partial }{\partial x_{i}}\left(\sum^{n}_{j=1}\sigma_{ij}(x)y_{j}I_{|y| < 1}(y)p(x, t)\right)\right]\nu(\,{\rm d}y).
\end{align}
\normalsize
Substituting (\ref{sp_8}), (\ref{sp_9})and  (\ref{sp_14}) into (\ref{sp_7}), we get
\begin{align}
   & \int_{\mathbb{R}^{d}}\varphi(x)\frac{\partial p(x, t)}{\partial t}\,{\rm d}x  \notag\\
  = &  -\int_{\mathbb{R}^{d}}\varphi(x)\sum^{d}_{i=1}\frac{\partial }{\partial x_{i}}\left[\left(f_{i}(x)  + \sum^{n}_{j=1}\sigma_{ij}(x)b_{j} + \frac{1}{2}\sum^{d}_{m=1}\sum^{n}_{j,l=1}\frac{\partial \sigma_{ij}(x)}{\partial x_{m}}\sigma_{ml}(x)A_{lj}\right)p(x, t)\right]\,{\rm d}x\notag\\
+ & \frac{1}{2}\int_{\mathbb{R}^{d}}\varphi(x)\sum^{d}_{i,m=1}\frac{\partial^{2}}{\partial x_{i}\partial x_{m}}\left(\sum^{n}_{j,l=1}\sigma_{ij}(x)\sigma_{ml}(x)A_{jl}\,p(x, t)\right)\,{\rm d}x\notag\\
+ & \int_{\mathbb{R}^{d}}\varphi(x)\,{\rm d}x\int_{\mathbb{R}^{n}\backslash \{0\}}\left[p(\tilde{H}(x, y), t) \left|\dfrac{\partial \tilde H(x, y)}{\partial x}\right| - p(x, t) + \sum^{d}_{i=1}\frac{\partial }{\partial x_{i}}\left(\sum^{n}_{j=1}\sigma_{ij}(x)y_{j}I_{|y| < 1}(y)p(x, t)\right)\right]\nu({\rm d}y).
\end{align}
The main result (\ref{s1_9})  follows from Lemma \ref{lemma2}.
\end{proof}

\section{Examples}
According to the main results presented in Theorem \ref{theorem1}, for a given SDE (\ref{s1_6}) driven by the L\'{e}vy process $L(t)$ with the generating triplet $(b, A, \nu)$, we only need solving the ODE (\ref{s1_8a}) for the   mapping $\tilde H$ to obtain the corresponding Fokker-Planck equation.

The following lemma, which can be verified easily from the definition, will be used in the examples to obtain the generating triplet of an $\mathbb{R}^n$-valued L\'{e}vy process given each of its component being a scalar L\'{e}vy process with known triplet and being independent of each other.
\begin{lemma}\label{lemma3}
Let $\{L_i (t), i=1, 2, \cdots, n\}$ be n independent scalar  L\'{e}vy processes with $L_i(t)$ having the generating triplet $(b_i, A_i, \nu_i)$. Then the $\mathbb{R}^n$-valued process $L(t)=(L_1(t), L_2(t), \cdots, L_n(t))$ is a $\mathbb{R}^n$-valued L\'{e}vy process with   the generating triplet $(b, A, \nu)$   and  $b=(b_1, b_2, \cdots, b_n)$, $A=\mathrm{diag}(A_1, A_2, \cdots, A_n)$, and $\nu({\rm d}x_1,{\rm d}x_2,\cdots,{\rm d}x_n ) = \sum\limits_{i=1}^{n}\left( \nu_i({\rm d}x_i) \left(\prod \limits_{k=1, k\ne i }^n  \delta_0( {\rm d}x_k)\right) \right)$, where $\delta_0(\cdot)$ is the Dirac measure on $\mathbb{R}$ centered at $0$.
\end{lemma}

\bigskip

\noindent\textbf{Example 1.  One-dimensional Marcus SDE driven by  $\alpha$-stable process}\\
Consider the following SDE,
\begin{equation}\label{ex1_1}
  {\rm d}X(t) = f(X(t)){\rm d}t + X(t)\diamond {\rm d}L (t), \qquad  X(0) = x_0 \in \mathbb{R},
\end{equation}
where $X(t) \in \mathbb{R}$, $f: \mathbb{R} \to \mathbb{R}$ is Lipschitz and continuously differentiable, $L(t)$ is a symmetry $\alpha$-stable process with the generating triplet $(b, A, \nu)$ as  $A=0$ and $\nu({\rm d} y) = \dfrac{{\rm d} y}{|y|^{1+\alpha}}$. Equation (\ref{ex1_1}) can be written in the form of (\ref{s1_6}) with $d=1$, $n=1$, and  $\sigma(X(t))=X(t)$.

Solving  (\ref{s1_8a}), which becomes
\begin{equation}\label{ex1_2}
  \frac{{\rm d}\Phi(r)}{{\rm d}r} = -\Phi(r)v, \qquad \Phi(0) = u,
\end{equation}
we get
\begin{equation}\label{ex1_3}
 \tilde H(u, v) = ue^{-v}.
\end{equation}
By substituting $b$, $A$, $\nu$, and $\tilde H$ into (\ref{s1_9}), we get the following Fokker-Planck equation which governs the probability density function of $X(t)$ in SDE (\ref{ex1_1}),
\begin{align}\label{ex1_6}
  \frac{\partial p(x, t)}{\partial t} &=  - \frac{\partial }{\partial x}\left[\left(f(x ) +   bx\right) p(x, t) \right]\nonumber\\
       &+ \int_{\mathbb{R} \backslash \{0\}}\frac{ p(xe^{ -y},t)e^{ -y}  - p(x, t)
   + \frac{\partial }{\partial x}(xp(x, t))y I_{|y| < 1}(y)}{|y|^{1+\alpha}}{\rm d}y.
\end{align}

\noindent\textbf{Example 2. Nonlinear oscillator under excitation of combined Gaussian and Poisson white noise} \\
\\
Consider the second-order SDE
\begin{eqnarray}\label{s2_1}
\ddot{Y}(t) + Y(t) = f(Y(t), \dot Y(t)) +\dot  B(t) + \dot Y(t)\diamond \dot C(t),
\end{eqnarray}
where $Y(t) \in \mathbb{R}$, $f: \mathbb{R}^2 \to \mathbb{R}$, $B(t)$ is a scalar Brownian motion, and $C(t)$ is a compound Poisson process expressed as
\begin{align}\label{s2_2}
C(t) = \sum\limits_{i=1}^{N(t)}  r_i.
\end{align}
Here $N(t)$ is a Poisson process with intensity parameter $\lambda$ and  $\{r_i, i=1, 2, \cdots, N(t)\}$  are i.i.d random variables with probability distribution function $\rho ({\rm d}x)$.  Note that (\ref{s2_1}) describes an stochastic oscillator \cite{SunDuanLi2016}. The first term in the left-hand side of (\ref{s2_1}) represents the acceleration, the second term represents elastic force, and the RHS represents all other forces. Let $X_1(t) = Y(t)$, $X_2(t) = \dot Y(t)$, then (\ref{s2_1}) can be written  as
\begin{eqnarray}\label{s2_3}
   \begin{cases}{\rm d}X_{1}(t) &=X_2(t) {\rm d}t \nonumber,\\
  {\rm d}X_{2}(t) &= -X_1(t) {\rm d}t + {\rm d} B(t) + X_2(t)\diamond {\rm d} C(t)\end{cases}.
\end{eqnarray}
 It can be expressed in the form of  (\ref{s1_6}) with $X(t)=\begin{pmatrix} X_1(t)\\X_2(t)\end{pmatrix}$, $f(X(t))~=~\begin{pmatrix} X_2(t)\\-X_1(t)\end{pmatrix}$,
$\sigma(X(t)) = \begin{pmatrix} 0 &0\\ 1 &X_2(t)\end{pmatrix}$,  and  $L(t) = \begin{pmatrix}B(t)\\C(t)\end{pmatrix}$. It follows from Lemma \ref{lemma3} that $L(t)$ is a two-dimensional L\'{e}vy process with the generating triplet $(b, A, \nu)$  as $b=0$, $A=\begin{pmatrix} 0 &0\\ 0 &1\end{pmatrix}$, and $\nu ({\rm d}y_1, {\rm d} y_2 ) = \lambda \delta_0({\rm d}y_1) \rho({\rm d}y_2)$.

Now the ODE (\ref{s1_8a}) becomes
\begin{eqnarray}\label{s2_4}
\begin{cases}
   \frac{d\Phi_{1}(r)}{{\rm d}r}& = 0,\\
   \frac{d\Phi_{2}(r)}{{\rm d}r}& =  -\Phi_{1}(r) - v_{2}\Phi_{2}(r)
   \end{cases}
\end{eqnarray}
with $\Phi(0) =u$.
It follows from (\ref{s2_4}) and (\ref{s1_7a}) that
\begin{align}\label{s2_8}
\tilde H((u_1,u_2)^T,(v_1, v_2)^T)=\begin{pmatrix} u_1\\ u_1 K(-v_2) + e^{-v_2} u_2\end{pmatrix},
\end{align}
where $K(\cdot)$ is defined as
\begin{align}\label{s2_6}
K(x)=\begin{cases}\dfrac{1-e^{x} }{x},\quad &{\text{for}} \quad x\ne 0\\ -1,\quad &{\text{for}} \quad x=0.\end{cases}
\end{align}
 Let $p(x_1, x_2, t)$ be the joint probability density function of $(X_1(t),X_2(t))$. From Theorem \ref{theorem1}, the governing equation for $p$ is
\begin{align}\label{s1_26}
   \frac{\partial p}{\partial t}(x_1, x_2, t)
   &=    -  \dfrac{\partial  }{\partial x_1}\left[x_2 p (x_1, x_2, t) \right]+   \dfrac{\partial }{\partial x_2}\left[\left(-x_1+\frac{1}{2}x_2\right) p  (x_1, x_2, t)\right] \nonumber\\
&+\frac{1}{2} \frac{\partial ^2 }{\partial x_2^2} \left[x_2^2 p(x_1,x_2,t)\right]\nonumber\\
     &+   \lambda\int_{\mathbb{R}^{1}\backslash \{0\}} [p(x_1,\, x_1 K(-y_2) + e^{-y_2} x_2,\, t)e^{-y_2} -  p(x_1, x_2, t)]\rho( {\rm d} y_2).
\end{align}

\textbf{Example 3.} We consider the  SDEs in $\mathbb{R}^{2}$.
\begin{align}\label{s3_1}
\begin{cases}
  dX_{1}(t) &=  f_{1}(X_{1}(t), X_{2}(t))dt + X_{2}(t)\diamond dL_{1}(t),\\
  dX_{2}(t) &=  f_{2}(X_{1}(t), X_{2}(t))dt + X_{1}(t)\diamond dL_{2}(t)
\end{cases}
\end{align}
with initial value $X_{0} = x_0 \in \mathbb{R}^{2}$,
where   $L_1(t)$ and $L_2(t)$ are two independent $\alpha$-stable processes with generating triplets as $(1, 0, \nu_1)$ and $(1, 0, \nu_2)$, respectively. $\nu_1$ and $\nu_2$ are defined as $\nu_1({\rm d}y)= \dfrac{{\rm d}y}{|y|^{1+\alpha_1}}$ and $\nu_2({\rm d}y)= \dfrac{{\rm d}y}{|y|^{1+\alpha_2}}$.
Note that (\ref{s3_1}) can be expressed in the form of (\ref{s1_6}) with $\sigma(x)=\begin{pmatrix} x_2 &0\\0 &x_1\end{pmatrix}$ and $L(t) = \begin{pmatrix}L_1(t)\\L_2(t)\end{pmatrix}$. By Lemma 3,  $L(t)$ is a two dimensional L\'{e}vy  process with the generating triplet as $b=\begin{pmatrix} 1\\1\end{pmatrix}$, $A=\begin{pmatrix} 0 &0\\0 &0\end{pmatrix}$ and $\nu({\rm d} y_1, {\rm d}y_2)=  \delta_0({\rm d}y_1)\nu_2({\rm d}y_2)+ \delta_0({\rm d}y_2)\nu_1({\rm d}y_1)$.

Solving ODE (\ref{s1_8a}), which now becomes
\begin{align}\label{s3_2}
\begin{cases}
   \frac{d\Phi_{1}(r)}{dr}& =- v_{1}\Phi_{2}(r),\\
   \frac{d\Phi_{2}(r)}{dr}& =- v_{2}\Phi_{1}(r).
   \end{cases}
\end{align}
with $(\Phi_1(0),\Phi_2(0) ) = (u_1,u_2)$, we get the   mapping $\tilde H$ as
\begin{align}
\tilde{H}((u_1, u_2)^T, (v_1, v_2)^T) =
\left(
  \begin{array}{c}
   u_1 \overline{\rm {cos}}(v_1v_2) - v_1 u_2  \overline{{\rm{sin}}}(v_1v_2)\\
     -v_2u_1 \overline{\rm {sin}}(v_1v_2) +   u_2 \overline{\rm {cos}}(v_1v_2)\\
  \end{array}
\right)
\end{align}
where $\overline{\rm {cos}}(\cdot)$ and $\overline{\rm {sin}}(\cdot)$ are functions defined as
\begin{align}
\overline{\rm {cos}}(x)=\begin{cases} \cosh(\sqrt{|x|}), \quad &{\text{for}} \quad x \ge 0,\\ \cos(\sqrt{|x|}),\quad &{\text{for}} \quad x <0,\end{cases}
\end{align}
and
\begin{align}
\overline{\rm {sin}}(x)=\begin{cases}\dfrac{\sinh(\sqrt{|x|})}{\sqrt{|x|}},\quad &{\text{for}} \quad x > 0, \\1,&{\text{for}} \quad x = 0,\\ \dfrac{\sin(\sqrt{|x|})}{\sqrt{|x|}}\quad &{\text{for}} \quad x <0\end{cases}
\end{align}

By substituting $b$, $A$, $\nu$, and $\tilde H$ into (\ref{s1_9}), we get the following Fokker-Planck equation for the probability density function $p(x_1,x_2,t)$ of the solution $(X_1(t), X_2(t))^T$ in SDE (\ref{s3_1}) as
\begin{align}
  &\frac{\partial p(x_1, x_2, t)}{\partial t}  = -\frac{\partial }{\partial x_{1}}\left[\left(f_{1}(x_1, x_2)  + x_{2}\right)p(x_1, x_2, t)\right]\notag\\
  &\quad - \frac{\partial }{\partial x_{2}}\left[\left(f_{2}(x_1, x_2)  + x_{1}\right)p(x_1, x_2, t)\right]\notag\\
   &\quad+ \int_{\mathbb{R} \backslash {0}}\left[  p(x_1, x_2-y_2x_1, t)   - p(x_1, x_2, t) + \frac{\partial }{\partial x_{2}}(x_{1}p(x_1, x_2, t)y_{2}I_{|y_2| < 1}(y_2))\right]\frac{{\rm d}y_2}{|y_2|^{1+\alpha_2}} \notag\\
  &\quad+ \int_{\mathbb{R} \backslash {0}}\left[ p(x_1-y_1x_2, x_2, t)   - p(x_1, x_2, t) + \frac{\partial }{\partial x_{1}}(x_{2}p(x_1, x_2, t)y_{1}I_{|y_1| < 1}(y_1))\right]\frac{{\rm d}y_1}{|y_1|^{1+\alpha_1}} .\notag\\
\end{align}


\begin{thebibliography}{10}

\bibitem{Klebaner2005}
F.~C. Klebaner.
\newblock {\em Introduction to {S}tochastic {C}alculus with {A}pplications}.
\newblock 2nd Edition, Imperial College Press, 2005.

\bibitem{Oksendal2003}
B.~K. Oksendal.
\newblock {\em Stochastic Differential Equations : an Introduction with
  Applications}.
\newblock Springer, 6th Edition, 2003.

\bibitem{Applebaum2009}
D.~Applebaum.
\newblock {\em L\'{e}vy Processes and Stochastic Calculus}.
\newblock Cambridge University Press, 2nd Edition, 2009.

\bibitem{LinCai2004}
Y.~K. Lin and G.~Q. Cai.
\newblock {\em Probabilistic Structural Dynamics: Advanced Theory and
  Applications}.
\newblock Springer, 2005.

\bibitem{Marcus1978}
S.~I. Marcus.
\newblock Modeling and analysis of stochastic differential equations driven by
  point processes.
\newblock {\em IEEE Transanctions on Information Theory}, IT-24:164--172, 1978.

\bibitem{Marcus1981}
S.~I. Marcus.
\newblock Modeling and approximation of stochastic differential equations
  driven by semimartingales.
\newblock {\em Stochastics}, 4:223--245, 1981.

\bibitem{KurtzPardouxProtter1995}
T.~Kurtz, E.~Pardoux, and P.~Protter.
\newblock Stratonovich stochastic differential equations driven by general
  martingales.
\newblock {\em Ann. Inst. Henri Poincare Prob. Stat.}, 31:351--377, 1995.

\bibitem{SunDuanLi2013}
X.~Sun, J.~Duan, and X.~Li.
\newblock An alternative expression for stochastic dynamical systems with
  parametric poisson white noise.
\newblock {\em Probabilistic Engineering Mechanics}, 32:1--4, 2013.

\bibitem{DiPaolaFalsone1993}
M.~Di~Paola and G.~Falsone.
\newblock It\^{o} and stratonovich integrals for delta-correlated processes.
\newblock {\em Probabilistic Engineering Mechanics}, 8:197--208, 1993.

\bibitem{DiPaolaFalsone1993b}
M.~Di~Paola and G.~Falsone.
\newblock Stochastic dynamics of non-linear systems driven by non-normal
  delta-correlated processes.
\newblock {\em ASME Journal of Applied Mechanics}, 60:141--148, 1993.

\bibitem{KanazawaSagawaHayakawa2012}
K.~Kanazawa, T.~Sagawa, and H.~Hayakawa.
\newblock Stochastic energetics for non-gaussian processes.
\newblock {\em Physical Review Letters}, 108:201601, 2012.

\bibitem{WanZhu2014}
J.~Wan and W.~Q. Zhu.
\newblock Stochastic averaging of quasi-integrable and non-resonant hamiltonian
  systems under combined gaussian and poisson white noise excitations.
\newblock {\em Nonlinear Dynamics}, 76:1271--1289, 2014.

\bibitem{KloedenPlaten1992}
P.~Kloeden and E.~Platen.
\newblock {\em Numerical Solutions of Stochastic Differential Equations}.
\newblock Springer, 1992.

\bibitem{Schertzer2001}
D.~Schertzer, M.~Larcheveque, J.~Duan, V.~Yanovsky, and S.~Lovejoy.
\newblock Fractional {F}okker-{P}lanck equation for nonlinear stochastic
  differential equations driven by non-{G}aussian {L}\'{e}vy stable noises.
\newblock {\em Journal of Mathematical Physics}, 42:200--212, 2001.

\bibitem{SunDuan2012}
X.~Sun and J.~Duan.
\newblock {F}okker-{P}lanck equations for nonlinear dynamical systems driven by
  non-gaussian {L}\'{e}vy processes.
\newblock {\em Journal of Mathematical Physics}, 53:072701, 2012.

\bibitem{Risken1996}
H.~Risken.
\newblock {\em The {F}okker-{P}lanck Equation: Methods of Solution and
  Applications}.
\newblock Springer, 1996.

\bibitem{ZengZhu2010c}
Y.~Zeng and W.~Q. Zhu.
\newblock Stochastic averaging of n-dimensional non-linear dynamical systems
  subject to non-{G}aussian wide-band random excitations.
\newblock {\em International Journal of Non-linear Mechanics}, 45:572--586,
  2010.

\bibitem{PriolaZabczyk2009}
E.~Priola and J.~Zabczyk.
\newblock Densities for {O}rnstein-{U}hlenbeck processes with jumps.
\newblock {\em Bull. London Math. Soc.}, 41:41--50, 2009.

\bibitem{BodnarchukKulik2011}
S.~V. Bodnarchuk and A.~M. Kulik.
\newblock Conditions of smoothness for the distribution density of a solution
  of a multidimensional linear stochastic differential equation with {L}\'{e}vy
  noise.
\newblock {\em Ukrainian Mathematical Journal}, 63:501--5055, 2011.

\bibitem{SongZhang2015}
Y.~Song and X.~Zhang.
\newblock Regularity of density for {SDE}s driven by degenerate {L}\'{e}vy
  noise.
\newblock {\em Electron. J. Probab.}, 20:1--27, 2015.

\bibitem{Chicone2006}
C.~Chicone.
\newblock {\em Ordinary Differential Equations with Applications, Second
  Edition}.
\newblock Springer-Verlag, 2006.

\bibitem{Bhattacharyya2012}
P.~Bhattacharyya.
\newblock {\em Distributions: Generalized Functions with Applications in
  Sobolev Spaces}.
\newblock De Gruyter, 2012.

\bibitem{SunDuanLi2016}
X.~Sun, J.~Duan, and X.~Li.
\newblock Stochastic modeling of nonlinear oscillators under combined
  {G}aussian and {P}oisson white noise: a viewpoint based on the energy
  conservation law.
\newblock {\em Nonlinear Dynamics}, 84:1311--1325, 2016.

\end{thebibliography}
\end{document}